\documentclass[12pt]{article}
\usepackage{amssymb,amsfonts,amsthm}
\usepackage{amsmath,amsthm,amssymb}
\usepackage{amscd}
\usepackage{amsfonts}
\usepackage{color}
\newcommand{\con}{\operatorname{con}}

\newcommand{\occ}{\operatorname{occ}}

\newcommand{\os}{\operatorname{OccSet}}

\newtheorem{theorem}{Theorem}[section]
\newtheorem{ex}[theorem]{Example}
\newtheorem{cor}[theorem]{Corollary}

\newtheorem{lemma}[theorem]{Lemma}
\newtheorem{definition}[theorem]{Definition}

\newtheorem{prop}[theorem]{Proposition}
\newtheorem{obs}[theorem]{Observation}

\newtheorem{question}{Question}

\newtheorem{claim}{Claim}

\begin{document}

\title{Non-finitely related and finitely related monoids}
\author{Olga  B. Sapir}

\date{Dedicated to professor Mikhail V. Volkov in honor of his 70th birthday}

\maketitle

\begin{abstract} We transform the method of Glasson into a sufficient condition
under which a monoid is non-finitely related, add a new member to the collection of
interlocking word-patterns,
and use it to show that the monoid
$M(ab^2a, a^2b^2)$ is non-finitely related.

We also give a sufficient condition under which a monoid is finitely related and use it show that 
$M(a^2b^2)$ is finitely related. Together with the results of Glasson this
completes the description of  all finitely related monoids among the monoids 
of the form $M(W)$ where every word ${\bf u} \in W$ depends on two variables
and every variable occurs twice in $\bf u$.

\end{abstract}

\section{Introduction}
\label{sec: intr}

A finite semigroup is {\em finitely related} if its term functions are determined by a finite set of finitary relations.
The study of finitely related semigroups started in 2011 with Davey et al. \cite{DJPS}, who showed among many other things that the finite relatedness is a varietal property, in the sense that finite algebras that satisfy the same identities are simultaneously finitely related or not.

The first example of a non-finitely related semigroup was found by Mayr \cite{Mayr}. It was the 
six-element Brandt monoid which behaves badly with respect to almost any varietal property. The second example \cite{Steindl}
was also a monoid: a semigroup with an identity element. In the light of a long history of studying the finite basis property of finite semigroups,
it is not surprising that both examples are monoids.
 Accidentally, but again not surprisingly, Steindl in \cite{Steindl} used implicitly the so-called {\em chain} word-patterns.

The chain words were initially introduced in 2010 by Wen Ting Zhang for showing that certain 5 element monoid $P_2^1$ is 
{\em hereditary finitely based} (HFB) (a monoid $M$ is HFB if every monoid in the variety generated by $M$ is finitely based).
$P_2^1$ is the only existing example of an HFB finite monoid which generates variety with infinitely many infinite (ascending) chains of subvarieties.
The monoid $P_2^1$ is contained in one of the limit varieties discovered by Sergey Gusev and the result of Zhang later became a part of \cite{Gusev-Li-Zhang}.

Since then the chain words have had an amazing variety of applications.
In 2014 Edmond Lee and Wen Ting Zhang used the chain words for constructing non-finitely generated  monoid varieties 
\cite{Lee14, Lee15, LZ}. In 2015 the chain words were used by Jackson and Lee to show that the monoid $M(abab)$ generates a variety with
uncountably many subvarieties \cite{JL}, and by  Jackson for  constructing a finite monoid with infinite irredundant identity basis \cite{MJ}.  Ren, Jackson, Zhao, and Lei  utilized chain words in the identification
of semiring limit varieties \cite{RJZL}. Independently, the chain words were introduced in \cite{Gusev-Vernikov}. Recently, Gusev \cite{SG} weaved the chain words into a more complicated pattern to show that the finitely based monoid $M(a^2b^2)$ also generates a variety with uncountably many subvarieties.

Glasson \cite{DG} used the chain words 
to prove that the monoid $M(abab)$ is not finitely related. 
Also in \cite{DG}, Glasson introduced two new word-patterns: the crown words to show that the monoid $M(ab^2a)$ is non-finitely related and generates a variety with uncountably many subvarieties \cite{DG1}, and 
 the maelstrom words to show that  $M(abab, a^2b^2)$ is non-finitely related. 

 In Sect. \ref{sec: S} we transform the method of Glasson  into a sufficient condition
under which a monoid is non-finitely related.
 In Sect. \ref{sec: W} we add a new word-pattern to the collection of interlocking words,
and use it to prove that  the monoid
$M(ab^2a, a^2b^2)$ is also non-finitely related.
 In Sect. \ref{sec: FR} we give a sufficient condition
under which a monoid is finitely related and use it show that 
$M(a^2b^2)$ is finitely related. Together with the results in \cite{DG} this
completes the description of  all finitely related monoids among the monoids 
of the form $M(W)$ where every word ${\bf u} \in W$ depends on two variables
and every variable occurs twice in $\bf u$.

\section{Preliminaries}

\subsection{Isoterms and Dilworth-Perkins construction}

A word $\bf u$ said to be an isoterm for a semigroup $S$ if $S$ does not satisfy any nontrivial identity of the form ${\bf u} \approx {\bf v}$.
For any set of words $W$, let $M(W)$ denote the Rees quotient of the free monoid over the ideal consisting of words which are not
subwords of any word in $W$. 
This construction was used by Perkins \cite{P} to build one of the first two examples of non-finitely based finite semigroups.
If $W = \{{\bf u}_1, \dots, {\bf u}_m\}$ for some $m>0$ then we write  $M(W) = M({\bf u}_1, \dots, {\bf u}_m).$

\begin{lemma} \cite[Lemma~3.3] {MJ05} \label{L: iso}

Let $W$ be a set of words and $M$ be a monoid. Then every word in $W$ is an isoterm for $M$ if and only if 
the variety generated by $M$ contains $M(W)$.

\end{lemma}

Given $\alpha, \beta \ge 1$ denote 
\[A_{\alpha, \beta}:= \{x^\alpha \approx x^{\alpha+\beta},  t_1 x t_2 x \dots t_\alpha x \approx x^\alpha t_1 t_2 \dots t_\alpha \}.\]

It is easy to see that if every variable occurs in every word in $W$ less than $\alpha$ times then $M(W)$ satisfies $A_{\alpha,1}$.

\subsection{The order $<_{\bf u}$ on the occurrence set of $\bf u$}

Given a word ${\bf u}$, we use $\con({\bf u})$ to denote the set of all variables appearing in $\bf u$.
Given $x \in \con({\bf u})$, we use $\occ(x, {\bf u})$ to denote the number of times $x$ occurs in $\bf u$.
We use ${_ix}$ to denote the $i$-th occurrence of a variable $x$ from the left. Define
\[\os({\bf u}) = \{ {_ix} \mid x \in \con({\bf u}), 1 \le i \le \occ(x, {\bf u}) \}.\] 
The word $\bf u$ induces a total order $<_{\bf u}$ on $\os({\bf u})$ defined by  $({_ix}) <_{\bf u} ({_jy})$ if and only if 
the $i^{th}$ occurrence of $x$ precedes the $j^{th}$ occurrence of $y$ in $\bf u$. We write $({_ix}) \ll_{\bf u} ({_jy})$
when ${_ix}$ immediately precedes  ${_jy}$ in $\bf u$.

\subsection{Schemes and Finite Relatedness}

We assume that the order of variables in a word ${\bf u} = {\bf u}(x_1, \dots, x_n)$ with $\con({\bf u}) = \{x_1, \dots, x_n\}$ is fixed and sometimes say that ${\bf u}$ is an $n$-ary term. For example, if ${\bf u}(x,y) = xyxy$ then ${\bf u}(y,x) = yxyx$.
Given $1 \le i < j \le n$ and ${\bf u}(x_1, \dots, x_n)$ we use ${\bf u}^{(ij)}$ to denote the result of renaming $x_i$ by $x_j$ in $\bf w$.

If $n \ge 2$ and $k=n-1$, then Definition~2.3 in \cite{DJPS} says that an indexed family
of $n$-ary terms $\{{\bf t}_{ij}(x_1, \dots, x_n) \mid 1 \le i < j \le n \}$ is a {\em scheme} for a semigroup $S$ if
for each $1 \le i < j \le n$ and $1 \le k< l \le n$ the 
following holds.

Dependency Condition:
\[S \models  {\bf t}_{ij}(x_1, \dots, x_i, \dots, x_j, \dots, x_n)  \approx   {\bf t}_{ij}(x_1, \dots, x_j, \dots, x_j, \dots, x_n).\]

Consistency Condition: 

{\bf Case 1}: If $i, j, k, l$ are pairwise distinct then $S \models  {\bf t}_{ij}^{(kl)} \approx  {\bf t}^{(ij)}_{kl}$;

{\bf Case 2}: If $j=k$  then $S \models  {\bf t}_{ij}^{(jl)} \approx  {\bf t}^{(il)}_{jl}$;

{\bf Case 3}: If $i=k$ and $j <l$  then $S \models  {\bf t}_{ij}^{(jl)} \approx  {\bf t}^{(jl)}_{il}$.

A scheme $\{{\bf t}_{ij}(x_1, \dots, x_n) \mid 1 \le i < j \le n \}$ for $S$ comes from a term ${\bf t}_n$ if
for each $1 \le i < j \le n$ we have 
$S \models  {\bf t}_n^{(ij)} \approx  {\bf t}_{ij}$.

\begin{lemma} \cite[Theorem~2.9] {DJPS} \label{L: nfr}

A semigroup $S$ is finitely related if and only if for sufficiently large $k$, every scheme $\{{\bf t}_{ij}(x_1, \dots, x_n) \mid 1 \le i < j \le n \}$ for $S$ 
with $n>k$ comes from a term.

\end{lemma}


\section{Sufficient condition under which a monoid is not finitely related} \label{sec: S}

Given variables $\{x, y, z\}$ and a word(term) $\bf u$, we use
${\bf u}^{\not x}$ to denote the result of deleting all occurrences of $x$ in $\bf u$, and ${\bf u}[y, z]$ to denote the result of deleting all occurrences of
all variables other than $y$ and $z$  in $\bf u$.
For $n \ge 2$ and $\mathcal X_n= \mathcal X (x_1, x_2, \dots,  x_n)$, define 
\[\mathcal Y (x_2, \dots,  x_n):= \mathcal X^{\not{x_1}} (x_1, x_2, \dots,  x_n).\]

The following theorem explains the method used by Glasson in \cite{DG} to prove 
Theorems 4.4, 4.13 and 4.18 and can be used to reprove them\footnote{Theorem~4.18 in \cite{DG} is missing \eqref{del22} for $\mathcal X = \mathcal R$.}.

\begin{theorem} \label{main}
 Let $M$ be a monoid.  
Suppose that  one can find a sequence of $n$-ary words $\{\mathcal X_n = \mathcal X(x_1, \dots, x_n)  \mid n =1, 2, \dots  \}$ such that 
 for some  (possibly partial) operation $*$ on the free monoid we have:

(i) for every $n \in \mathbb N$, if $r =0, 2, 4, \dots$ then $M$ satisfies
\begin{equation} \label{del1} \mathcal X^{\not{y_1}} (x_1, \dots,x_r, y_1, \dots, y_n)  \approx  \mathcal X(x_1, \dots, x_r) * \mathcal Y(y_2, \dots, y_n);\end{equation}
 if $r =1, 3, 5, \dots$ then $M$ satisfies
\begin{equation} \label{del22} \mathcal X^{\not{y_1}} (x_1, \dots,x_r, y_1, \dots, y_n)  \approx  \mathcal X(x_1, \dots, x_r) * \mathcal X (y_2,  \dots, y_n);\end{equation}

(ii) for every $r, k \in \mathbb N$, $M$ satisfies
\begin{equation} \label{com}  \mathcal X(x_1, \dots, x_r) * \mathcal X(y_1, \dots, y_k)  \approx \mathcal X(y_1, \dots, y_k)   * \mathcal X(x_1, \dots, x_r); \end{equation}

(iii) $\occ(x_1, \mathcal X_2) = \occ(x_2, \mathcal X_2) = d$ and $M \models A_{\alpha,\beta}$ for some $\alpha, \beta \in \mathbb N$ with  $\alpha \le 2d$;

(iv)  for any even $n\ge 6$ if ${\bf t}_n={\bf t}(x_1, \dots, x_n)$ is an $n$-ary word such that  $M \models {\bf t}_n[x_i,x_{j}] \approx \mathcal X_n[x_i,x_{j}]$
for each $\{i,j\} \ne \{1,n\}$, then $M \not \models {\bf t}_n[x_1,x_n] \approx \mathcal Y(x_n, x_1)$.

Then $M$ is not finitely related. 

\end{theorem}

If we delete $x_1$ from both sides of \eqref{del1} then we obtain that for every even $r$, $M$ satisfies:
\begin{equation} \label{del4} \mathcal Y^{\not{y_1}} (x_2, \dots,x_r, y_1, \dots, y_n)  \approx  \mathcal Y(x_2, \dots, x_r) * \mathcal Y(y_2, \dots, y_n).\end{equation}
 
If we delete $x_1$ from both sides of \eqref{del22} then we obtain  that for every odd $r$, $M$ satisfies:
\begin{equation} \label{del3} \mathcal Y^{\not{y_1}} (x_2, \dots,x_r, y_1, \dots, y_n)  \approx  \mathcal Y(x_2, \dots, x_r) * \mathcal X (y_2,  \dots, y_n).\end{equation}

For each even $n \ge 2$ and $1 \le m \le n$ define \[\mathcal Z(x_{m+1}, \dots, x_n, x_1, \dots, x_{m-1}):=  \mathcal  X (x_{m+1}, \dots, x_n, x_1, \dots, x_{m-1})\] if $m$ is even and
\[\mathcal Z(x_{m+1}, \dots, x_n, x_1, \dots, x_{m-1}):=  \mathcal  Y (x_{m+1}, \dots, x_n, x_1, \dots, x_{m-1})\] if $m$ is odd.

\begin{lemma} \label{lemma} 

 Let $M$ be a monoid  that satisfies (i)--(iii) in Theorem~\ref{main}.
Fix even $n\ge 2$. Then the family of terms \[\{{\bf t}_{ij}=   \mathcal Z^{\not{x_j} }(x_{i+1}, \dots, x_n, x_1, \dots, x_{i-1}) x_j^\alpha \ \mid 1 \le i < j \le n\}\] is a scheme for $M$.

\end{lemma}

\begin{proof}  
First, let us verify that for every even $n$ and $1 \le s \le n, 1 \le  m \le n$, $M$ satisfies
\begin{equation} \label{ms} \mathcal Z^{\not{x_m} }(x_{s+1}, \dots, x_n, x_1, \dots, x_{s-1})  \approx \mathcal  Z^{\not{x_s} } (x_{m+1}, \dots, x_n, x_1, \dots, x_{m-1}).\end{equation}

{\bf Case 1}: If both $s$  and $m$ are odd, we need to verify that
\[M \models \mathcal Y^{\not{x_m} }(x_{s+1}, \dots, x_n, x_1, \dots, x_{s-1})  \approx \mathcal  Y^{\not{x_s} } (x_{m+1}, \dots, x_n, x_1, \dots, x_{m-1}).\]
It enough to assume that $s=1$ and $1 < m < n$.

\[M \models \mathcal Y^{\not{x_m} }(x_2, \dots, x_n) \stackrel{\eqref{del4}}{\approx} \]
\[ \mathcal Y(x_2, \dots, x_{m-1}) * \mathcal Y(x_{m+1}, \dots, x_n)   \stackrel{\eqref{com}}{\approx}   \]
\[\mathcal Y(x_{m+1}, \dots, x_n)* \mathcal Y( x_2, \dots, x_{m-1})  \stackrel{\eqref{del4}}{\approx}   \]
\[\mathcal  Y^{\not{x_1} } (x_{m+1}, \dots, x_n, x_1, x_2, \dots, x_{m-1}).\]

{\bf Case 2}: If both $s$  and $m$ are even, we need to verify that
\[M \models \mathcal X^{\not{x_m} }(x_{s+1}, \dots, x_n, x_1, \dots, x_{s-1})  \approx \mathcal  X^{\not{x_s} } (x_{m+1}, \dots, x_n, x_1, \dots, x_{m-1}).\]
It enough to assume that $s=2$ and $2 < m \le n$.

\[M \models \mathcal X^{\not{x_m} }(x_3, \dots, x_n, x_1)  \stackrel{\eqref{del22}}{\approx} \]
\[ \mathcal X (x_3, \dots, x_{m-1}) * \mathcal X (x_{m+1}, \dots, x_n, x_1)   \stackrel{\eqref{com}}{\approx}   \]
\[\mathcal X (x_{m+1}, \dots, x_n, x_1) * \mathcal X ( x_3, \dots, x_{m-1})  \stackrel{\eqref{del22}}{\approx}   \]
\[\mathcal  X^{\not{x_2} } (x_{m+1}, \dots, x_n, x_1, x_2, x_3, \dots, x_{m-1}).\]

{\bf Case 3}: If  $s$  is odd and $m$ is even, we need to verify that
\[\mathcal Y^{\not{x_m} }(x_{s+1}, \dots, x_n, x_1, \dots, x_{s-1})  \approx \mathcal  X^{\not{x_s} } (x_{m+1}, \dots, x_n, x_1, \dots, x_{m-1}).\]

If we  assume that $s=1$ and $1 < m \le n$ then:
\[M \models \mathcal Y^{\not{x_m} }(x_2, \dots, x_n)  \stackrel{\eqref{del4}}{\approx} \]
\[ \mathcal Y(x_2, \dots, x_{m-1}) * \mathcal X (x_{m+1}, \dots, x_n)   \stackrel{\eqref{com}}{\approx}   \]
\[\mathcal X (x_{m+1}, \dots, x_n) * \mathcal Y( x_2, \dots, x_{m-1})  \stackrel{\eqref{del1}}{\approx}   \]
\[\mathcal  X^{\not{x_1} } (x_{m+1}, \dots, x_n, x_1, x_2, \dots, x_{m-1}).\]

If we assume that $m=2$ and $2 < s < n$ then:
\[M \models \mathcal X^{\not{x_s} }(x_3, \dots, x_n, x_1)  \stackrel{\eqref{del1}}{\approx} \]
\[ \mathcal X (x_3, \dots, x_{s-1}) * \mathcal Y(x_{s+1}, \dots, x_n, x_1)   \stackrel{\eqref{com}}{\approx}   \]
\[\mathcal Y(x_{s+1}, \dots, x_n, x_1) * \mathcal X (x_3, \dots, x_{s-1}) \stackrel{\eqref{del3}}{\approx}   \]
\[\mathcal  Y^{\not{x_2} } (x_{s+1}, \dots, x_n, x_1, x_2, x_3, \dots, x_{s-1}).\]

Now we use $A_{\alpha, \beta}$ and \eqref{ms} to verify that for each even $n \ge 2$ the family of terms \[\{{\bf t}_{ij}=   \mathcal Z^{\not{x_j} }(x_{i+1}, \dots, x_n, x_1, \dots, x_{i-1}) x_j^\alpha \ \mid 1 \le i < j \le n\}\] is a scheme for $M$.
Indeed, the dependency condition is trivially satisfied. To verify the consistency consider the three cases.

{\bf Case 1}: If $i, j, k, l$ are pairwise disjoint then
 \[M \models {\bf t}^{(kl)}_{ij} \approx   \mathcal Z^{\not{x_j}, \not{x_k}, \not{x_l}} (x_{i+1}, \dots, x_n, x_1, \dots, x_{i-1}) x_j^\alpha x_l^\alpha
\stackrel{\eqref{ms}}{\approx} \] 
\[\mathcal Z^{\not{x_l}, \not{x_i}, \not{x_j}}(x_{k+1}, \dots, x_n, x_1, \dots, x_{k-1}) x_l^\alpha x_j^\alpha \approx  {\bf t}^{(ij)}_{kl}.\]

{\bf Case 2}: If $j=k$ then   ${\bf t}_{jl} =  \mathcal Z^{\not{x_l} } (x_{j+1}, \dots, x_n, x_1, \dots, x_{j-1}) x_l^\alpha$.
Hence
\[M \models {\bf t}^{(jl)}_{ij} \approx   \mathcal Z^{\not{x_j}, \not{x_l} } (x_{i+1}, \dots, x_n, x_1, \dots, x_{i-1}) x_l^\alpha \stackrel{\eqref{ms}}{\approx} \] 
\[ \mathcal Z^{\not{x_l}, \not{x_i}} (x_{j+1}, \dots, x_n, x_1, \dots, x_{j-1}) x_l^\alpha \approx  {\bf t}^{(il)}_{jl}.\]

{\bf Case 3}: If $i=k$ and $j<l$ then  ${\bf t}_{il} =  \mathcal X^{\not{x_l} }\{x_{i+1}, \dots, x_n, x_1, \dots, x_{i-1}\} x_l^\alpha$.
 Hence
\[M \models {\bf t}^{(jl)}_{ij} \approx   \mathcal Z^{\not{x_j}, \not{x_l} } (x_{i}, \dots, x_n, x_1, \dots, x_{i-1}) x_l^\alpha \approx  {\bf t}^{(jl)}_{il}.\]
\end{proof}

\begin{obs} \label{O} 
Let $M$ be a monoid, $n \in \mathbb N$ and $\mathcal X_n = \mathcal X(x_1, \dots, x_n)$
be an $n$-ary word such that 
Condition (i) of Theorem~\ref{main} is satisfied. Then:

(i) for each odd $1\le i < n$, we have \[M \models \mathcal X_n [x_i, x_{i+1}] \approx \mathcal X(x_i, x_{i+1}),\]
 and for each even $2 \le i < n$, we have
\[M \models \mathcal X_n [x_i, x_{i+1}] \approx \mathcal Y(x_i, x_{i+1});\]

(ii) If $\occ(x_1, \mathcal X_2) = \occ(x_2, \mathcal X_2) = d$ then
for each $i \in \mathbb N$ we have $\occ(x_i, \mathcal X_n)=d$ and 
 for each  $1 \le i < j \le n$ with $j - i \ge 2$ we have 
\[M \models \mathcal X_n [x_i, x_{j}] \approx  x_i^d * x_j^d.\]

\end{obs}

\begin{proof}[Proof of Theorem~\ref{main}]
Take even $n\ge 6$. 
Suppose that  the $n$-ary scheme \[\{{\bf t}_{ij}=   \mathcal Z^{\not{x_j} }(x_{i+1}, \dots, x_n, x_1, \dots, x_{i-1}) x_j^\alpha \ \mid 1 \le i < j \le n\}\] comes from a term ${\bf t}_n ={\bf t} (x_1, \dots, x_n)$ for $V(M)$, that is,
\[M \models {\bf t}_n^{(sm)} \approx  \mathcal Z^{\not{x_m} }(x_{s+1}, \dots, x_n, x_1, \dots, x_{s-1}) x_m^\alpha \] for each $1 \le s<m \le n$.
In particular, 
\[M \models {\bf t}_n^{(n-1, n)} \approx  \mathcal X(x_1, \dots, x_{n-2}) x_n^\alpha.\] 
Using  \eqref{del22}  to delete extra variables, we obtain
\[M \models {\bf t}_n[x_i, x_{i+1}] \approx  \mathcal X(x_i, x_{i+1})  \stackrel{Obs.~\ref{O}}{\approx}
\mathcal X_n[x_i, x_{i+1}] \]
for each $i=1, 3, \dots, n-3$.

Using  \eqref{del1}  to delete extra variables, we obtain
\[M \models {\bf t}_n[x_i, x_{i+1}] \approx  \mathcal Y(x_i, x_{i+1})  \stackrel{Obs.~\ref{O}}{\approx}
\mathcal X_n[x_i, x_{i+1}] \]
for each $i=2, 4, \dots, n-4$.

Also, 
\[M \models {\bf t}_n^{(1, 2)} \approx  \mathcal X(x_3, \dots, x_{n}) x_2^\alpha.\] 
Using  \eqref{del22}  to delete extra variables, we obtain
\[M \models {\bf t}_n[x_{n-1}, x_n] \approx  \mathcal X(x_{n-1}, x_{n})  \stackrel{Obs.~\ref{O}}{\approx}
\mathcal X_n[x_{n-1}, x_{n}].\]
Using  \eqref{del1}  to delete extra variables, we obtain
\[M \models {\bf t}_n[x_{n-2}, x_{n-1}] \approx  \mathcal Y(x_{n-2}, x_{n-1})  \stackrel{Obs.~\ref{O}}{\approx}
\mathcal X_n[x_{n-2}, x_{n-1}].\]

Now fix $2 \le i < j \le n$ with $j-i \ge 2$. 
Then  
\[M \models {\bf t}^{(i-1, i+1)} \approx  \mathcal Z^{\not{x_{i+1}} }(x_{i}, x_{i+1}, x_{i+2}, \dots, x_n, x_1, \dots, x_{i-2}) x_{i+1}^\alpha  \stackrel{\eqref{del22} or \eqref{del4}}{\approx}\]
\[[\mathcal Z(x_i) * \mathcal Z(x_{i+2}, \dots, x_n, x_1, \dots, x_{i-2})] x_{i+1}^\alpha.\]
If we delete everything except for $\{x_i, x_j\}$  from both sides then we obtain
\[M \models {\bf t}[x_i,x_j] \approx  x^d_i * x^d_j  \stackrel{Obs.~\ref{O}}{\approx}
\mathcal X_n[x_{i}, x_{j}].\]

Finally,  
\[M \models {\bf t}_n^{(2, n)} \approx  \mathcal X^{\not{x_n}}(x_3, x_4, \dots, x_{n-1}, x_n, x_1) x_n^\alpha
 \stackrel{\eqref{del1}}{\approx}\]
\[ [\mathcal X(x_{3}, x_4 \dots, x_{n-1})* \mathcal X(x_1)] x_{n}^\alpha.\]
Deleting extra variables we obtain
 \[M \models {\bf t}[x_1, x_j] \approx x_j^d * x_1^d \stackrel{\eqref{com}}{\approx}  x_1^d * x_j^d  \stackrel{Obs.~\ref{O}}{\approx}
\mathcal X_n[x_1, x_{j}]\]
for each $3 \le j \le n-1$.

Overall, we conclude that
 $M \models {\bf t}_n[x_i,x_{j}] \approx \mathcal X_n[x_i,x_{j}]$ for each $\{i,j\} \ne \{1,n\}$.

Now consider 
\[M \models {\bf t}_n^{(n-2, n-1)} \approx  \mathcal Y(x_n,  x_1, x_2, \dots, x_{n-3}) x_{n-1}^\alpha.\] 
Using  \eqref{del1}  to delete extra variables, we obtain
\[M \models {\bf t}_n[x_1 ,x_n] \approx \mathcal Y(x_n, x_1).\]
Since this contradicts Condition (iv), such a term ${\bf t}_n$ is impossible. Therefore, the scheme does not come from any $n$-ary term, and $M$ is non-finitely related
by Lemma~\ref{L: nfr}.
\end{proof}


\section{Chain, crown, maelstrom and other interlocking words}
\label{sec: W}

To show how new interlocking words $\mathcal S_n$ are similar
to the chain $\mathcal C_n$, maelstrom $\mathcal M_n$ and crown words $\mathcal R_n$,
we define all these words recursively as follows.

\begin{definition} \label{D: words}
 $\mathcal C_0 = \mathcal M_0 = \mathcal R_0 =  \mathcal S_0 = 1, \mathcal C_1 = \mathcal M_1 = \mathcal R_1 =  \mathcal S_1 = x_1^2$.

\begin{itemize}

\item{If $\mathcal C_k = {\bf c}_k \cdot x_k$ then $\mathcal C_{k+1} = {\bf c}_k \cdot x_{k+1} \cdot  x_k \cdot x_{k+1}$.}

\item{ If $k$ is odd and $\mathcal M_k = {\bf m}_k \cdot x_k$ then $\mathcal M_{k+1} = x_{k+1}\cdot {\bf m}_k \cdot x_{k+1} \cdot x_k$;

If $k$ is even and $\mathcal M_k =x_{k} \cdot {\bf m}'_k$ then $\mathcal M_{k+1} = x_k \cdot x_{k+1}\cdot {\bf m}'_k \cdot x_{k+1}$.}

\item{ If $k$ is odd and $\mathcal R_k = {\bf r}_k \cdot x_k$ then $\mathcal R_{k+1} = {\bf r}_k \cdot x^2_{k+1} \cdot  x_k$;

If $k$ is even and $\mathcal R_k = {\bf r}'_k \cdot  x^2_k x_{k-1}$ then $\mathcal R_{k+1} = {\bf r}'_k \cdot x_{k+1} \cdot x^2_k x_{k-1} \cdot x_{k+1}$.}

\item{ If $k$ is odd and $\mathcal S_k = (x_{k-1} x_k \cdot {\bf s}_k) ( x_k \cdot  {\bf s}'_k)$ then \\ $\mathcal S_{k+1} =( x_{k+1} \cdot x_{k-1} x_k \cdot {\bf s}_k) ( x_k \cdot  x_{k+1} \cdot {\bf s}'_k)$;

If $k$ is even and $\mathcal S_k = (x_k \cdot {\bf s}'_k) (x_{k-1}  x_k  \cdot  {\bf s}_k)$ then \\ $\mathcal S_{k+1} =   (x_k  \cdot x_{k+1} \cdot {\bf s}'_k )  (x_{k+1}  \cdot  x_{k-1}    x_k \cdot {\bf s}_k)$.}

\end{itemize}
\end{definition}

For example,
\[\mathcal S_8 = ( x_8 x_6 x_7 x_4 x_5 x_2 x_3 x_1) ( x_7 x_8 x_5 x_6 x_3 x_4 x_1 x_2).\]

We say that a word $\bf u$ is {\em double-linear} if  ${\bf u} = (x_1 \dots x_n)(x_{\sigma(1)}, \dots, x_{\sigma(n)})$ for some $n>0$ and a permutation $\sigma$ on
$\{1, \dots, n\}$. Notice that maelstrom $\mathcal M_n$ and the new words $\mathcal S_n$ are double-linear. 
For the rest of this note we fix $*$ to be the following operation on double-linear words:

$\bullet$ If ${\bf u}$ and $\bf v$ are double linear then ${\bf u} * {\bf v}: = {\bf u}_1 {\bf v}_1{\bf u}_2{\bf v}_2$, where
${\bf u}_1$, ${\bf u}_2$, ${\bf v}_1$ and ${\bf v}_2$ are linear words such that  ${\bf u} = {\bf u}_1 {\bf u}_2$ and  ${\bf v} = {\bf v}_1 {\bf v}_2$.

Notice that  $\mathcal S_{k+1}$ is obtained from $\mathcal S_{k}$ by inserting  ${{_1x}_{k+1}}$ before the first linear part of  $\mathcal S_{k}$
and inserting  ${{_2x}_{k+1}}$ immediately to the right of the first letter in the second linear part of  $\mathcal S_{k}$ when $k$ is odd, and inserting them vice versa when $k$ is even. Having this in mind, we give three more definitions to the words $\mathcal S_n$.
First, recall that \[\mathcal Y (x_2, \dots,  x_n):= \mathcal S^{\not{x_1}} (x_1, x_2, \dots,  x_n).\]

\begin{prop} \label{P: S} Let $n \ge 3$ and $\mathcal S_n = \mathcal S(x_1, \dots, x_n)$  be an $n$-ary word such that $\mathcal S_n [x_1, x_2, x_3]= (x_2 x_3 x_1)(x_3 x_1 x_2)$.
Then the following are equivalent:

(i)  
\[\mathcal S^{\not{x_m}} (x_1, \dots,  x_n)  =  \mathcal S(x_1, \dots, x_{m-1}) * \mathcal Y (x_{m+1}, \dots, x_n)\] for each odd  $1 \le m \le n$, and
 \[\mathcal S^{\not{x_m}} (x_1, \dots,  x_n)  =  \mathcal S(x_1, \dots, x_{m-1})  * \mathcal S (x_{m+1}, \dots, x_n)\] for each even  $1 \le m \le n$.

(ii) for each odd $1 \le i < n$ we have $\mathcal S_n[x_i, x_{i+1}] = x_{i+1} x^2_{i} x_{i+1}$, for each
even $2 \le i < n$ we have $\mathcal S_n[x_i, x_{i+1}] = x_{i} x^2_{i+1} x_{i}$,
and for each $1 \le i < j \le n$ with $j-i \ge 2$ we have $\mathcal S_n[x_i, x_{j}] = x_j^2 * x_i^2 = x_j x_i x_j x_i$.

(iii) for each odd $1 \le i < n$ we have 
$\mathcal S_n[x_i, x_{i+1}] = x_{i+1} x^2_{i} x_{i+1}$, for each
even $2 \le i < n$ we have $\mathcal S_n[x_i, x_{i+1}] = x_{i} x^2_{i+1} x_{i}$,
for each $1 \le i < j \le n$ with $j-i =2$ or  $j-i =3$ we have $\mathcal S_n[x_i, x_{j}]  \in \{ x_i x_j x_i x_j, x_j x_i x_j x_i\}$,
and for each odd $3 \le j \le n$ we have  $\mathcal S_n[x_1, x_{j}]  \in \{ x_1 x_j x_1 x_j, x_j x_1 x_j x_1\}$.

(iv) $\mathcal S_n$ is as in Definition~\ref{D: words}.
\end{prop}

\begin{proof} 
 Clearly, the statement holds for $n=3$. Suppose that it holds for each $3 \le n \le k$. Take $n=k+1$. 

Implication (i) $\rightarrow$(ii) is by  Observation~\ref{O} (take $M$ to be the free monoid).

Implication (ii)$\rightarrow$(iii) is evident.

(iii)$\rightarrow$(iv) By induction hypothesis, $\mathcal S_k$  is as in Definition~\ref{D: words}.

If $k$ is odd then $\mathcal S_k$ begins with $x_{k-1} x_k$ and the second linear part of
$\mathcal S_k$ begins with $x_{k} x_{k-2}$.
Since $\mathcal S_{k+1}[x_k, x_{k+1}] = x_{k+1} x^2_{k} x_{k+1}$, either
 ${_1x}_{k+1}$ is the first in $\mathcal S_{k+1}$ or we have
\[({_1x}_{k-1}) \ll_{\bf u} ({_1x}_{k+1}) \ll_{\bf u} ({_1x}_{k}),\]
where ${\bf u} =\mathcal S_{k+1}$. In both cases we must have
 $\mathcal S_{k+1}[x_{k-2}, x_{k+1}] = x_{k+1} x_{k-2} x_{k+1} x_{k-2}$.
Hence the only possibility for the second occurrence of $x_{k+1}$ is:
\[({_2x}_{k}) \ll_{\bf u} ({_2x}_{k+1}) \ll_{\bf u} ({_2x}_{k-2}),\]
where ${\bf u} =\mathcal S_{k+1}$. 
Then we must have $\mathcal S_{k+1}[x_{k-1}, x_{k+1}] = x_{k+1} x_{k-1} x_{k+1} x_{k-1}$.
Therefore, the letter ${_1x}_{k+1}$ must be the first in $\mathcal S_{k+1}$.
The resulting word is $\mathcal S_{k+1}$ as in  Definition~\ref{D: words}.

If $k$ is even then $\mathcal S_k$ begins with $x_{k} x_{k-2}$ and
the second linear part of $\mathcal S_k$ begins with $x_{k-1}x_k$.
Since $\mathcal S_{k+1}[x_k, x_{k+1}] = x_{k} x^2_{k+1} x_{k}$ 
and
$\mathcal S_{k+1}[x_{k-2}, x_{k+1}] = x_{k+1} x_{k-2} x_{k+1} x_{k-2}$,
the only possibility for the first occurrence of $x_{k+1}$ is:
\[({_1x}_{k}) \ll_{\bf u} ({_1x}_{k+1}) \ll_{\bf u} ({_1x}_{k-2}),\]
where ${\bf u} =\mathcal S_{k+1}$. Then $\mathcal S_{k+1}[x_{k-1}, x_{k+1}] = x_{k+1} x_{k-1} x_{k+1} x_{k-1}$.
Therefore, the second occurrence of $x_{k+1}$ must 
precede  the second occurrence of $x_{k-1}$. Then
 $\mathcal S_{k+1}[x_{1}, x_{k+1}] = x_{k+1} x_{1} x_{k+1} x_{1}$.
This leaves only one possibility for the second occurrence of $x_{k+1}$:
\[({_1x}_{1}) \ll_{\bf u} ({_2x}_{k+1}) \ll_{\bf u} ({_2x}_{k-1}),\]
where ${\bf u} =\mathcal S_{k+1}$. 
The resulting word $\mathcal S_{k+1}$ as in  Definition~\ref{D: words}.

Implication (iv)$\rightarrow$(i) is a routine.
 \end{proof}

Let $\bar{*}$ denote the operation on the double-linear words which is dual to the operation $*$.
Here are  four definitions of  words $\bar{\mathcal S}_n$.

\begin{definition} \label{D: S1} Let $n \ge 3$ and $\bar{\mathcal S}_n = \bar{\mathcal S}(x_1, \dots, x_n)$  be an $n$-ary word such that $\bar{\mathcal S}_n [x_1, x_2, x_3]= (x_2  x_1 x_3)(x_1 x_3 x_2)$.
Then the following are equivalent:

(i)  \[\bar{\mathcal S}^{\not{x_m}} (x_1, \dots,  x_n)  =  \bar{\mathcal S} (x_1, \dots, x_{m-1}) \bar{*} \mathcal Y (x_{m+1}, \dots, x_n)\] for each odd  $1 \le m \le n$, and \[\bar{\mathcal S}^{\not{x_m}} (x_1, \dots,  x_n)  =  \bar{\mathcal S} (x_1, \dots, x_{m-1})  \bar{*} \bar{\mathcal S} (x_{m+1}, \dots, x_n)\] for each even  $1 \le m \le n$.

(ii) for each odd $1 \le i < n$ we have $\bar{\mathcal S}_n[x_i, x_{i+1}] = x_{i+1} x^2_{i} x_{i+1}$, for each
even $2 \le i < n$ we have $\bar{\mathcal S}_n[x_i, x_{i+1}] = x_{i} x^2_{i+1} x_{i}$,
and for each $1 \le i < j \le n$ with $j-i \ge 2$ we have $\bar{\mathcal S}_n[x_i, x_{j}] = x_i^2 \bar{*} x_j^2 = x_i x_j x_i x_j$.

(iii) for each odd $1 \le i < n$ we have 
$\bar{\mathcal S}_n[x_i, x_{i+1}] = x_{i+1} x^2_{i} x_{i+1}$, for each
even $2 \le i < n$ we have $\bar{\mathcal S}_n[x_i, x_{i+1}] = x_{i} x^2_{i+1} x_{i}$,
for each $1 \le i < j \le n$ with $j-i =2$ or  $j-i =3$ we have $\bar{\mathcal S}_n[x_i, x_{j}] \in \{ x_i x_j x_i x_j, x_j x_i x_j x_i\}$
and for each odd $3 \le j \le n$ we have  $\bar{\mathcal S}_n[x_1, x_{j}]  \in \{ x_1 x_j x_1 x_j, x_j x_1 x_j x_1\}$.

(iv)  If $k$ is odd and $\bar{\mathcal S}_k = ({\bf s}'_k \cdot x_k) ({\bf s}_k  \cdot x_k x_{k-1})$   then    \\ 
$\bar{\mathcal S}_{k+1} = ({\bf s}'_k  \cdot x_{k+1} x_k)   ({\bf s}_k  \cdot x_k x_{k-1} \cdot x_{k+1})$;

If $k$ is even and $\bar{\mathcal S}_k = ({\bf s}_k \cdot  x_k   x_{k-1}) ({\bf s}'_k \cdot x_k)$ then\\ 
$\bar{\mathcal S}_{k+1} =   ({\bf s}_k  \cdot   x_k    x_{k-1}  \cdot  x_{k+1})  ({\bf s}'_k  \cdot x_{k+1} \cdot  x_k)$.

\end{definition}

Notice that Condition (iii) in Proposition~\ref{P: S} and Definition~\ref{D: S1} is the same.
The words $\bar{\mathcal S}_n$ compared to $\mathcal S_n$ 
are enumerated in the opposite direction. For example,
\[\bar{\mathcal S}_8 = ( x_2 x_1 x_4 x_3 x_6 x_5 x_8 x_7) ( x_1 x_3 x_2 x_5 x_4 x_7 x_6 x_8). \]

The following theorem is similar to Theorem~4.13 in \cite{DG}.

\begin{theorem}  \label{abba}
Let $M$ be a monoid that  satisfies $A_{4,\beta}$ for some $\beta \ge 1$ and
for every $r, k \in \mathbb N$ we have
\[ M \models \mathcal S(x_1, \dots, x_r) * \mathcal S(y_1, \dots, y_k)  \approx \mathcal S(y_1, \dots, y_k)  * \mathcal S(x_1, \dots, x_r),\]
where $\mathcal S_n = \mathcal S(x_1, \dots, x_n)$ is as in Definition~\ref{D: words}.

Suppose that $\mathcal S(x,y) = yx^2y$ is an isoterm for $M$ and the words $\{xyxy, yxyx\}$ can form an identity of $M$ only with each other.
Then $M$ is non-finitely related.
\end{theorem} 

\begin{proof} If $\mathcal X = \mathcal S$ then
Condition~(i)  of Theorem~\ref{main} is satisfied by Proposition~\ref{P: S}(i). Clearly $M$ satisfies  Condition (iii) of Theorem~\ref{main} with $d=2$.

Now fix even $n\ge 6$ and  let ${\bf t}(x_1, \dots, x_n)$ be an $n$-ary term such that  $M \models {\bf t}_n[x_i,x_{j}] \approx \mathcal S_n[x_i,x_{j}]$
for each $\{i,j\} \ne \{1,n\}$.
Then for each odd $1 \le i \le n-1$ we have $M \models {\bf t}_n[x_i, x_{i+1}] \approx  x_{i+1} x^2_i x_{i+1}$,  for each even $2 \le i \le n-2$ we have $M \models {\bf t}_n[x_i, x_{i+1}] \approx  x_i x^2_{i+1} x_{i}$.
Since $n>4$,  $M \models {\bf t}_n[x_i, x_{i+2}] \approx x_j x_i x_jx_i$ for each $1 \le i < j \le n$ with $j-i =2$ or  $j-i =3$. 
Also for each $j= 3, 5, \dots, n-1$ we have  $M \models {\bf t}_n[x_1, x_j] \approx x_j x_1 x_jx_1$.

 Since $xy^2x$ is an isoterm for $M$ and the words $\{xyxy, yxyx\}$ can form an identity of $M$ only with each other, 
 for each odd $1 \le i \le n-1$ we have ${\bf t}_n[x_i, x_{i+1}] =  x_{i+1} x^2_i x_{i+1}$, for each even $2 \le i \le n-2$ we have ${\bf t}_n[x_i, x_{i+1}] = x_{i+1} x^2_{i} x_{i+1}$, and for each $1 \le i < j \le n$ with $j-i =2$ or  $j-i =3$ 
 we have ${\bf t}_n[x_i, x_{j}] \in  \{x_i x_j x_ix_j, x_jx_i x_j x_i\}$.
Also for each $j= 3, 5, \dots, n-1$ we have  ${\bf t}_n[x_1, x_{j}]  \in \{ x_1 x_j x_1 x_j, x_j x_1 x_j x_1\}$.

If ${\bf t}_n[x_1, x_{3}] = x_3 x_1 x_3 x_1$, then 
${\bf t}_n[x_1, x_2, x_{3}] = \mathcal S_3= (x_2 x_3 x_1)(x_3 x_1 x_2)$.
Proposition~\ref{P: S}[(iii)$\rightarrow$(ii)] implies that  ${\bf t}_n = \mathcal S_n$ and ${\bf t}_n[x_1,x_n]= {\mathcal S}_n[x_1,x_n]=x_n x_1 x_nx_1$.

 If ${\bf t}_n[x_1, x_{3}] = x_1 x_3 x_1 x_3$, then 
${\bf t}_n[x_1, x_2, x_{3}] = \bar{\mathcal S}_3 =  (x_2 x_1 x_3) ( x_1 x_3 x_2)$.
Definition~\ref{D: S1}[(iii)$\rightarrow$(ii)]  implies that  ${\bf t}_n = \bar{\mathcal S}_n$ and ${\bf t}_n[x_1,x_n]= \bar{{\mathcal S}}_n[x_1,x_n]=x_1 x_n x_1x_n$.

Since $xy^2x$ is an isoterm for $M$, the monoid $M$ does not satisfy ${\bf t}_n[x_1,x_n] \approx \mathcal Y[x_n, x_1] = x_n x^2_1 x_n$.
Therefore, $M$ is non-finitely related by Theorem~\ref{main}.
\end{proof}

Theorem~\ref{abba} together with Lemma~\ref{L: iso} readily gives us the following.

\begin{ex} \label{E: aba}
The monoid $M(a^2b^2, ab^2a)$ is non-finitely related.

\end{ex}


\section{More definitions of chain and maelstrom words}

Notice that Proposition~\ref{P: S} and Definition~\ref{D: S1} play the key role in proving Theorem~\ref{abba}. In this section we give similar 
definitions to the chain and maelstrom words. These definitions can be used in a similar way to reprove Theorems~4.4 and 4.13 in \cite{DG}.

Let $\odot$ denote the `wrapping' operation in \cite[Definition~4.10]{DG}, but in the reverse order, that is, ${\bf w} \odot {\bf v}$ means that
$\bf v$ wraps around $\bf w$. 
Here are three more definitions of maelstrom words.

\begin{prop} \label{P: M} Let $n \ge 3$ and $\mathcal M_n = \mathcal M(x_1, \dots, x_n)$  be an $n$-ary word such that $\mathcal M_n [x_1, x_2, x_3]= (x_2 x_3 x_1)(x_2 x_1 x_3)$.
Then the following are equivalent:

(i)  
\[\mathcal M^{\not{x_m}} (x_1, \dots,  x_n)  =  \mathcal M(x_1, \dots, x_{m-1})  \odot \mathcal Y (x_{m+1}, \dots, x_n)\] if $1 \le m \le n$ is odd, and
 \[\mathcal M^{\not{x_m}} (x_1, \dots,  x_n)  =  \mathcal M(x_1, \dots, x_{m-1})  \odot \mathcal M (x_{m+1}, \dots, x_n)\] if $1 \le m \le n$ is even.

(ii) for each odd $1 \le i < n$ we have $\mathcal M_n[x_i, x_{i+1}] = x_{i+1} x_{i} x_{i+1} x_{i}$, for each
even $2 \le i < n$ we have $\mathcal M_n[x_i, x_{i+1}] = x_{i} x_{i+1} x_{i} x_{i+1}$,
and for each $1 \le i < j \le n$ with $j-i \ge 2$ we have $\mathcal M_n[x_i, x_{j}] = x_i^2 \odot x_{j}^2 = x_j x_i^2 x_j$.

(iii) for each odd $1 \le i < n$ we have $\mathcal M_n[x_i, x_{i+1}] = x_{i+1} x_{i} x_{i+1} x_{i}$, for each
even $2 \le i < n$ we have $\mathcal M_n[x_i, x_{i+1}] = x_{i} x_{i+1} x_{i} x_{i+1}$,
and for each $1 \le i < j \le n$ with $j-i =2$ or  $j-i =3$ we have $\mathcal M_n[x_i, x_{j}] \in  \{x_j x_i^2 x_j, x_i x_j^2 x_i\}$.

(iv) $\mathcal M_n$ is  as in Definition~\ref{D: words}.
\end{prop}

\begin{proof} As in Proposition~\ref{P: S} only one implication needs verification.

(iii)$\rightarrow$(iv) By induction hypothesis, $\mathcal M_k$ is a maelstrom word as in Definition~\ref{D: words}. If $k$ is odd then
 we need to insert $x_{k+1}$ into $\mathcal M_k$ so that
$\mathcal M_{k+1}[x_k, x_{k+1}] = x_{k+1} x_{k} x_{k+1} x_{k}$.
Then either ${_1x}_{k+1}$ is the first letter in $\mathcal M_{k+1}$ or
we have \[({_1x}_{k-1}) \ll_{\bf u} ({_1x}_{k+1}) \ll_{\bf u} ({_1x}_{k}),\]
where ${\bf u} =\mathcal M_{k+1}$. In both cases, we must have
$\mathcal M_{k+1}[x_{k-2}, x_{k+1}] = x_{k+1} x_{k-2}^2 x_{k+1}$.
This leaves only one possibility for ${_2x}_{k+1}$:
\[({_2x}_{k-2}) \ll_{\bf u} ({_2x}_{k+1}) \ll_{\bf u} ({_2x}_{k}).\]
Hence we must have $\mathcal M_{k+1}[x_{k-1}, x_{k+1}] = x_{k+1} x_{k-1}^2 x_{k+1}$.
This implies that   ${_1x}_{k+1}$ is the first letter in $\mathcal M_{k+1}$.
The resulting word $\mathcal M_{k+1}$ is maelstrom by  Definition~\ref{D: words}.

If $k$ is even then we need to insert $x_{k+1}$ into $\mathcal M_k$ so that
$\mathcal M_{k+1}[x_k, x_{k+1}] = x_{k} x_{k+1} x_{k} x_{k+1}$.
Then either ${_2x}_{k+1}$ is the last letter in $\mathcal M_{k+1}$ or
we have \[({_2x}_{k}) \ll_{\bf u} ({_2x}_{k+1}) \ll_{\bf u} ({_2x}_{k-1}),\]
where ${\bf u} =\mathcal M_{k+1}$. In both cases, we must have
$\mathcal M_{k+1}[x_{k-2}, x_{k+1}] = x_{k+1} x_{k-2}^2 x_{k+1}$.
This leaves only one possibility for ${_1x}_{k+1}$:
\[({_1x}_{k}) \ll_{\bf u} ({_1x}_{k+1}) \ll_{\bf u} ({_1x}_{k-2}).\]
Hence we must have $\mathcal M_{k+1}[x_{k-1}, x_{k+1}] = x_{k+1} x_{k-1}^2 x_{k+1}$.
This implies that   ${_2x}_{k+1}$ is the last letter in $\mathcal M_{k+1}$.
The resulting word $\mathcal M_{k+1}$ is maelstrom by  Definition~\ref{D: words}.
\end{proof}

Let $\bar{\odot}$ denote the operation on the double-linear words which is dual to the operation $\odot$
(this will be the operation in \cite[Definition~4.10]{DG}).
Here are  four definitions of words $\bar{\mathcal M}_n$.

\begin{definition} \label{D: M1} Let $n \ge 3$ and $\bar{\mathcal M}_n = \bar{\mathcal M}(x_1, \dots, x_n)$  be an $n$-ary word such that $\bar{\mathcal M}_n [x_1, x_2, x_3]= (x_2 x_1 x_3)(x_2 x_3 x_1)$.
Then the following are equivalent:

(i)  
\[\bar{\mathcal M}^{\not{x_m}} (x_1, \dots,  x_n)  =  \bar{\mathcal M}(x_1, \dots, x_{m-1})  \bar{\odot} \mathcal Y (x_{m+1}, \dots, x_n)\] if $1 \le m \le n$ is odd, and
 \[\bar{\mathcal M}^{\not{x_m}} (x_1, \dots,  x_n)  =  \bar{\mathcal M}(x_1, \dots, x_{m-1})  \bar{\odot} \bar{\mathcal M} (x_{m+1}, \dots, x_n)\] if $1 \le m \le n$ is even.

(ii) for each odd $1 \le i < n$ we have $\bar{\mathcal M}_n[x_i, x_{i+1}] = x_{i+1} x_{i} x_{i+1} x_{i}$, for each
even $2 \le i < n$ we have $\bar{\mathcal M}_n[x_i, x_{i+1}] = x_{i} x_{i+1} x_{i} x_{i+1}$,
and for each $1 \le i < j \le n$ with $j-i \ge 2$ we have $\bar{\mathcal M}_n[x_i, x_{j}] = x_i^2 \bar{\odot} x_{j}^2 = x_ix_j^2 x_i$.

(iii) for each odd $1 \le i < n$ we have $\bar{\mathcal M}_n[x_i, x_{i+1}] = x_{i+1} x_{i} x_{i+1} x_{i}$, for each
even $2 \le i < n$ we have $\bar{\mathcal M}_n[x_i, x_{i+1}] = x_{i} x_{i+1} x_{i} x_{i+1}$,
and for each $1 \le i < j \le n$ with $j-i =2$ or  $j-i =3$ we have $\bar{\mathcal M}_n[x_i, x_{j}] \in  \{x_j x_i^2 x_j, x_i x_j^2 x_i\}$.

(iv)  If $k$ is odd and $\bar{\mathcal M}_k = {\bf m}_k \cdot x_k x_{k-1} x_k \cdot {\bf m}'_k$ then \\ $\bar{\mathcal M}_{k+1} = {\bf m}_k \cdot x_{k+1}\cdot x_k \cdot x_{k+1} \cdot x_{k-1}  x_k \cdot {\bf m}'_k$;

If $k$ is even and $\bar{\mathcal M}_k = {\bf m}_k \cdot x_k x_{k-1} x_k \cdot {\bf m}'_k$ then \\ $\bar{\mathcal M}_{k+1} = {\bf m}_k \cdot  x_k  x_{k-1} \cdot  x_{k+1} \cdot x_k \cdot x_{k+1} \cdot {\bf m}'_k$.

\end{definition}

Again, as for $\mathcal S_n$ and $\bar{\mathcal S}_n$ (Proposition~\ref{P: S} and Definition~\ref{D: S1}),  Condition (iii) in Proposition~\ref{P: M} and Definition~\ref{D: M1} is the same.
The words $\bar{\mathcal M}_n$ can be obtained from $\mathcal M_n$ by switching the order of linear parts. 
For example,
\[\mathcal M_8 = ( x_8 x_6 x_7 x_4 x_5 x_2 x_3 x_1)( x_2 x_1 x_4 x_3 x_6 x_5 x_8 x_7)  \]
\[\bar{\mathcal M}_8 = ( x_2 x_1 x_4 x_3 x_6 x_5  x_8 x_7 )(x_8  x_6  x_7x_4  x_5 x_2 x_3 x_1 )  \]

Finally, three more definitions of the chain words.

\begin{prop} \label{P: C} 
Let $n \ge 3$ and $\mathcal C_n = \mathcal C(x_1, \dots, x_n)$  be an $n$-ary word such that $\mathcal C_n [x_1, x_2]= x_1 x_2 x_1 x_2$.
Then the following are equivalent:

(i) \[\mathcal C^{\not{x_m}} (x_1, \dots,  x_n)  =  \mathcal C(x_1, \dots, x_{m-1})  \mathcal  C(x_{m+1}, \dots, x_n)\] for any $1 \le m \le n$; 

(ii) for each $1 \le i \le n-1$ we have $\mathcal C_n[x_i, x_{i+1}] = x_i x_{i+1} x_i x_{i+1}$ and for each
$1 \le i < j \le n$ with $j-i \ge 2$ we have $\mathcal C_n[x_i, x_{j}] = x_i^2 x_{j}^2$.

(iii) for each $1 \le i \le n-1$ we have $\mathcal C_n[x_i, x_{i+1}] = x_i x_{i+1} x_i x_{i+1}$ and for each
$1 \le i < j \le n$ with $j-i = 2$ we have $\mathcal C_n[x_i, x_{j}] \ne  x_i x_j x_i x_j$.

(iv) $\mathcal C_n$ is  as in Definition~\ref{D: words}.

\end{prop}

\begin{proof}
Again, only one implication needs verification.

(iii)$\rightarrow$(iv) By induction hypothesis, $\mathcal C_k$ is a chain word as in Definition~\ref{D: words}.
Then there is only one possibility to insert $x_{k+1}$
in $\mathcal C_k$ so that $\mathcal C_{k+1}[x_k, x_{k+1}] = x_k x_{k+1} x_k x_{k+1}$ and $\mathcal C_{k+1}[x_{k-1}, x_{k+1}] \ne  x_{k-1} x_{k+1} x_{k-1} x_{k+1}$. 
The resulting word $\mathcal C_{k+1}$ is chain by  Definition~\ref{D: words}.
\end{proof}


\section{Sufficient condition under which a monoid is finitely related} \label{sec: FR}

Let  $X_n=\{x_1, \dots, x_n\}$ and $\mathcal F (X_n) = \{{\bf t}_{ij} \mid 1 \le i < j \le n \}$ be an $n$-ary scheme for $M$.
 For each $1 \le s \le n$ define
  \[\occ(x_s, \mathcal F) = \min \{\occ(x_s, {\bf t}_{ij}) \mid 1 \le i < j \le  n, s \not \in \{i, j \}\}.\]
In  \cite{DG, Steindl}, the value $\occ(x_s, \mathcal F)$ is referred to as the exponent of $x_s$ in scheme $\mathcal F$.
It is convenient  to say that  $\occ(x_s, \mathcal F)$ denotes the number of times $x_s$ occurs in $\mathcal F$.

\begin{lemma} \label{L: x^m}  Let $M$ be a finite monoid such that  for some $m > 0$, $M \models A_{m+1,1}$
and $x^m$ is an isoterm for $M$.  Let $n > \max(6, |M|+1)$ and $\mathcal F (X_n) = \{{\bf t}_{ij} \mid 1 \le i < j \le n \}$ be an $n$-ary scheme for $M$.
Then for each $1 \le s \le n$

(i) we may assume that  $1 \le \occ(x_s, \mathcal F) \le m$;

(ii) for every $1 \le i < j \le n$  with  $s \not \in \{i, j \}$  we have $\occ(x_s, {\bf t}_{ij}) = \occ(x_s, \mathcal F)$;
 
 (iii)   if $occ(x_r, \mathcal F) + occ(x_s, \mathcal F) \le m$  then \[occ(x_s, {\bf t}_{rs}) = occ(x_r, \mathcal F) + occ(x_s, \mathcal F)\]
 and $occ(x_s, {\bf t}_{rs}) >m$ otherwise.
\end{lemma}

\begin{proof} (i) In view of Lemma~2.3 in \cite{DG} we may assume
that $\con({\bf t}_{ij}) = X_n \setminus\{x_i\}$ for each $1 \le i < j \le  n$.
Hence $1 \le \occ(x_s, \mathcal F)$. In view of Lemma~2.6 in \cite{DG}
we may assume that $\occ(x_s, \mathcal F) \le m$.

Part (ii)  follows from Lemma~2.4(i) in \cite{DG} and the fact that $x^m$ is an isoterm for $M$. 
Part (iii) is by Lemma~2.4(ii) in \cite{DG}.
\end{proof}

For every $n$-ary scheme $\mathcal F (X_n)$  
define \[\os(\mathcal F) = \{ {_ix}_s \mid 1 \le s \le n, 1 \le i \le \occ(x_s, \mathcal F) \}.\] 
For every $1 \le p \le c$ and $1 \le q \le d$ define
\[C_{(p \le c, q \le d)} =  \{ \{{_px}, {_qy}\} \subset  \os^2(\mathcal F)\mid \occ(x, \mathcal F) = c, \occ(y, \mathcal F)=d \}.\]
For  $\{{_px}, {_qy}\} \in C_{(p \le c, q \le d)}$  define

$\bullet$  $({_px}) <_{\mathcal F} ({_qy})$ if
 for each $1 \le k < l \le n$   with $\{x,y\} \cap \{x_k, x_l \} = \emptyset$ we have $({_{p}x}) <_{({\bf t}_{kl})} ({_{q}y})$.

Given $1 \le p \le c$ and $1 \le q \le d$ we say that 
an  identity  ${\bf u} \approx {\bf v}$  satisfies property $P_{(p \le c, q \le d)}$ 
if for each $x, y \in \con({\bf u})$ with $\occ(x, {\bf u}) = \occ(x, {\bf v}) = c$ and
 $\occ(y, {\bf u}) = \occ(y, {\bf v}) = d$ we have
$({_{p}x}) <_{\bf u} ({_{q}y})$ iff $({_{p}x}) <_{\bf v} ({_qy})$.

\begin{lemma} \label{L: order}  Let $M$ be a finite monoid such that  for some $m > 0$, $M \models A_{m+1,1}$
and $x^m$ is an isoterm for $M$.  Let $n > \max(6, |M|+1)$ and $\mathcal F (X_n) = \{{\bf t}_{ij} \mid 1 \le i < j \le n \}$ be an $n$-ary scheme for $M$
as in Lemma~\ref{L: x^m}. Suppose that for some  $1 \le p \le c \le m$ and $1 \le q \le d \le m$
every identity of $M$ satisfies  Property $P_{(p \le c, q \le d)}$.
Then 

 (i)  for every $\{{_px}, {_qy}\} \in C_{(p \le c, q \le d)}$ either
 $({_px}) <_{\mathcal F} ({_qy})$ or $({_qy}) <_{\mathcal F} ({_px})$;

(ii) the relation  $<_{\mathcal F}$ is a partial order on $\os(\mathcal F)$ with domain $C_{(p \le c, q \le d)}$.

\end{lemma}

\begin{proof} (i)  Since $\occ(x, \mathcal F) = c$ and $\occ(y, \mathcal F)=d$, Lemma~\ref{L: x^m} implies that
for each $1 \le k < l \le  n$ with $\{x, y\} \cap \{x_k, x_l \} =\emptyset$ we have $\occ(x, {\bf t}_{kl}) =c$ and $\occ(y, {\bf t}_{kl}) =d$.

Suppose that for some $1 \le i < j \le  n$ with $\{x, y\} \cap \{x_i, x_j \} =\emptyset$ we have  $({_{p}x}) <_{({\bf t}_{ij})} ({_{q}y})$.
Since $M \models  {\bf t}_{ij}^{(kl)} \approx  {\bf t}^{(ij)}_{kl}$  and  every identity of $M$ satisfies  Property $P_{(p \le c, q\le d)}$, we have $({_{p}x}) <_{({\bf t}_{kl})} ({_{q}y})$.
This means that $({_qx}) <_{\mathcal F} ({_py})$.

(ii) If we restrict the relation $\langle C_{(p \le c, q \le d)}, <_{\mathcal F}\rangle$ to a connected subset of $\os(\mathcal F)$ with at most three elements, then for some $1 \le k < l \le  n$, it coincides with the total order $<_{({\bf t}_{kl})}$ on this subset. Hence it is anti-symmetric and transitive.
\end{proof}

\begin{obs}  \label{O: P} Suppose that for some  $1 \le p \le c $ and $1 \le q \le d$
every identity of a semigroup $S$ satisfies  Property $P_{(p \le c, q \le d)}$.
Then every identity of  $S$ also satisfies  Property $P_{(p-u \le c-u-v, q \le d)}$ for each $0 \le u \le p-1$ and $0 \le v \le c-p$.
\end{obs}

\begin{proof}  To obtain a contradiction, suppose that   $S \models {\bf u} \approx {\bf v}$ such that for some
$x, y \in \con({\bf u})$ with $\occ(x, {\bf u}) = \occ(x, {\bf v}) = c-u-v$ and $\occ(y, {\bf u}) = \occ(y, {\bf v}) = d$, we have
$({_{p-u}x}) <_{\bf u} ({_{q}y})$ but  $({_qy}) <_{\bf v} ({_{p-u}x})$.
Then $S \models x^u \cdot {\bf u}\cdot x^{v} \approx x^u \cdot {\bf v} \cdot x^{v}$. Since the second identity does not satisfy Property $P_{(p \le c, q \le d)}$,
this contradicts our assumption that every identity of $S$ satisfies  Property $P_{(p \le c, q \le d)}$.
\end{proof}

Given $m>0$ we say that an identity  ${\bf u} \approx {\bf v}$ is {\em $m$-balanced} if $\con({\bf u}) = \con({\bf v})$ and for each $x \in \con({\bf u})$ we have $\occ(x, {\bf u}) = \occ(x, {\bf v})$
whenever either $\occ(x, {\bf u})\le m$ or $\occ(x, {\bf u}) \le m$.
It is easy to see that
every identity of a monoid $M$ is $m$-balanced  if and only if the word $x^m$ is an isoterm for $M$.

\begin{theorem} \label{T: main2} Let $M$ be a finite monoid such that  for some $m > 0$, $M \models A_{m+1,1}$
and $x^m$ is an isoterm for $M$.   Suppose that for some index set $I$ of quadruples,
one can find a (possibly empty) collection of properties $\{P_{(p \le c, q \le d)} \mid (p, c,  q, d) \in I\}$ such that
the following holds:

$\bullet$ $M \models {\bf u} \approx {\bf v}$ if and only if
 ${\bf u} \approx {\bf v}$ is $m$-balanced and  satisfies every property $P_{(p \le c, q \le d)}$ in the collection.

Let $n > \max(6, |M|+1)$ and $\mathcal F (X_n) = \{{\bf t}_{ij} \mid 1 \le i < j \le n \}$ be an $n$-ary scheme for $M$
as in Lemma~\ref{L: x^m}.  Let  $<_{\mathcal F}$ be the partial order on  $\os(\mathcal F)$ with domain  $\biguplus\limits_{ (p, c,  q, d) \in I} C_{(p \le c, q \le d)}$
as explained in Lemma~\ref{L: order}.

If $<_{\mathcal F}$ can be extended to a total order on $\os(\mathcal F)$, then $M$ is finitely related. 

\end{theorem}

\begin{proof} Let $<_{\bf t}$ denote the extension of $<_{\mathcal F}$  to a total order on $\os(\mathcal F)$.
Define  the word  ${\bf t}$ by ordering  the set $\os(\mathcal F)$ under $<_{\bf t}$. Then
 $\os(\mathbf t) = \os(\mathcal F)$, $\con({\bf t}) = X_n$ and $\occ(x, {\mathcal F}) = \occ(x, {\bf t})$ for each $x \in X_n$.

Take some $1 \le i < j \le n$. 
Since  $\con({\bf t}) = X_n$, we have $\con({\bf t}^{(ij)}) = \con({\bf t}_{ij}) = X_n \setminus \{x_i\}$. 
Let us verify that the identity ${\bf t}^{(ij)} \approx  {\bf t}_{ij}$ is $m$-balanced. Indeed,  if $occ(x_i, \mathcal F) + occ(x_j, \mathcal F) \le m$ then
\[ \occ(x_j,  {\bf t}_{ij}) \stackrel{Lemma~\ref{L: x^m}}{=}   occ(x_i, \mathcal F) + occ(x_j, \mathcal F)  =  occ(x_i, \mathbf t) + occ(x_j, \mathbf t)  =  \occ(x_j,  {\bf t}^{(ij)}), \]
Otherwise,   $\occ(x_j,  {\bf t}_{ij}) >m$ by Lemma~\ref{L: x^m} and $\occ(x_j,  {\bf t}^{(ij)}) >m$ too.

If $1 \le s \le n$ such that $s \not\in \{i,j\}$ then    
\[\occ(x_s,  {\bf t}^{(ij)}) = \occ(x_s,  {\bf t}) =  \occ(x_s,  {\mathcal F}) = \occ(x_s,  {\bf t}_{ij}).\]

 If the collection of properties $\{P_{(p \le c, q \le d)} \mid (p, c,  q, d) \in I\}$ is empty, then $M$ is equationally equivalent to $M(a^m)$.
Since the monoid  $M(a^m)$ is commutative it is finitely related by Theorem~3.6 in \cite{DJPS}.
So, we may assume that the collection of properties is not empty.
Let us verify that  the identity ${\bf t}^{(ij)} \approx  {\bf t}_{ij}$ satisfies every property in $\{P_{(p \le c,q \le d)} \mid (p, c,  q, d) \in I\}$.

Indeed, if $(p, c,  q, d) \in I$ then $1 \le p \le c\le m$, $1 \le q \le d\le m$ and  for every identity  ${\bf u} \approx {\bf v}$ of $M$ the following holds:

$\bullet$  for each $x, y \in \con({\bf u})$ with $\occ(x, {\bf u})=c$ and $\occ(y, {\bf u})=d$ we have
\[({_{p}x}) <_{\bf u} ({_{q}y}) \Leftrightarrow ({_{p}x}) <_{\bf v} ({_{q}y}).\]

Take  some $x, y \in X_n$ such that  $\occ(x, {\bf t}_{ij})=c$ and $\occ(y, {\bf t}_{ij})=d$.
If $\{x_i, x_j\} \cap \{x, y\} = \emptyset$ then
\[({_{p}x}) <_{{\bf t}^{(ij)}} ({_{q}y}) \Leftrightarrow ({_{p}x}) <_{{\bf t}} ({_{q}y}) \Leftrightarrow ({_px}) <_{\mathcal F} ({_qy}) \stackrel{Lemma~\ref{L: order}}{\Leftrightarrow}({_{p}x}) <_{{\bf t}_{ij}} ({_{q}y}).\]

Now assume that say, $x=x_j$. Since $n\ge 6$ we can pick some $1 \le k < l \le n$ with $\{x_k, x_l\} \cap \{x_i, x_j\} = \emptyset$. 

\begin{claim} \label{c}
If $({_{p}x_j}) <_{{\bf t}_{kl}^{(ij)}} ({_{q}y})$ then  $({_{p}x_j}) <_{{\bf t}^{(ij)}} ({_{q}y})$.
 \end{claim}

\begin{proof}[Proof of claim]  If $({_{p}x_j}) <_{{\bf t}_{kl}^{(ij)}} ({_{q}y})$ then 
for some $0 \le u \le p$ and $0 \le v \le c-p$   we have
  $({_{p-u}x_j}) <_{{\bf t}_{kl}} ({_{q}y})$, $\occ(x_j, {\bf t}_{(kl)})=c-u -v$,
 $({_ux_i}) <_{{\bf t}_{kl}} ({_{q}y})$ and $\occ(x_i, {\bf t}_{(kl)})= u+v$.   
 By Observation~\ref{O: P} the monoid $M$ satisfies Properties  $P_{(p-u \le c-u-v, q \le d)}$ and  $P_{(u \le u+v, q \le d)}$.
Then
 \[({_{p-u}x_j}) <_{{\bf t}_{kl}} ({_{q}y}) \stackrel{Lemma~\ref{L: order}}{\Leftrightarrow} ({_{p-u}x_j}) <_{\mathcal F} ({_{q}y}) \Leftrightarrow ({_{p-u}x_j}) <_{\mathbf t} ({_{q}y});\]
 \[({_{u}x_i}) <_{{\bf t}_{kl}} ({_{q}y}) \stackrel{Lemma~\ref{L: order}}{\Leftrightarrow} ({_{u}x_i}) <_{\mathcal F} ({_{q}y}) \Leftrightarrow ({_{u}x_i}) <_{\mathbf t} ({_{q}y}).\]
Consequently, $({_{p}x_j}) <_{{\bf t}^{(ij)}} ({_{q}y})$.
\end{proof}

Since
$M \models  {\bf t}_{ij}^{(kl)} \approx  {\bf t}^{(ij)}_{kl}$, we have 
\[ ({_{p}x_j}) <_{{\bf t}_{ij}} ({_{q}y}) \Leftrightarrow ({_{p}x_j}) <_{{\bf t}_{ij}^{(kl)}} ({_{q}y})     
 \Leftrightarrow ({_{p}x_j}) <_{{\bf t}_{kl}^{(ij)}} ({_{q}y})  \stackrel{Claim~\ref{c}}{\Leftrightarrow}  ({_{p}x_j}) <_{{\bf t}^{(ij)}} ({_{q}y}). \]
 
 Therefore $M \models  {\bf t}^{(ij)} \approx  {\bf t}_{ij}$. 
This means that $\mathcal F(X_n)$ comes from the term ${\bf t}(X_n)$ and $M$ is finitely related by Lemma~\ref{L: nfr}.
\end{proof}

 Let $W_m$ denote the set of all words involving at least two variables, where each 
variable occurs at most $m$ times. Then the identities of  $M(W_n)$ satisfy every property $P_{(p \le c, q \le d)}$ for each $1 \le p \le c \le m$ and $1 \le q \le d \le m$.
Hence the relation  $<_{\mathcal F}$ is a total order on $\os(\mathcal F)$. Thus Theorem~\ref{T: main2} readily implies Theorem~3.2 in \cite{DG} which says that
the monoid $M(W_m)$ is finitely related.

\begin{ex} \label{E: aabb}
The monoid $M=M(a^2b^2)$ is finitely related.
\end{ex}

\begin{proof} Notice that  $x^2$ is an isoterm for $M$ and $M \models A_{3, 1}$.
Take $n>11$ and consider an $n$-ary scheme  $\mathcal F (X_n) = \{{\bf t}_{ij} \mid 1 \le i < j \le n \}$ for $M$.
 In view of Lemma~\ref{L: x^m} we have:
\[X_n = T \uplus X \] where $T$ is the set of all linear variables in $\mathcal F$ and $X$ is the set of all 2-occurring variables in $\mathcal F$.
Then \[\os({\mathcal F}) = T \uplus  {_1X} \uplus {_2X},\] where  ${_1X}$ is the set of all first occurrences of variables in $X$, and
 ${_2X}$ is the set of all second occurrences of variables in $X$.

It is easy to see that every identity of $M$  satisfies Property $P_{(1 \le 2, 2 \le 2)}$. Using Observation~\ref{O: P}, it is routine to verify
the following:

 $\bullet$ $M \models {\bf u} \approx {\bf v}$ if and only if
 ${\bf u} \approx {\bf v}$ is 2-balanced and satisfies
 \[P_{(1 \le 1, 1\le 1)} \wedge P_{(1 \le 1, 1 \le 2)} \wedge  P_{(1 \le 1, 2 \le 2)} \wedge    P_{(1 \le 2, 2 \le 2)}.\]
This means that  the order $<_{\mathcal F}$ is defined between any two letters in $T$ and for every pair $\{t, {_1x}\}$, $\{t, {_2x}\}$, $\{{_1x}, {_2y}\}$ for each $t \in T$ and $x, y \in X$. The order $<_{\mathcal F}$ is not defined only between the first occurrences of letters in $X$ and  between the second occurrences of letters in $X$.

For each ${_1x} \in \os({\mathcal F})$ define
\[L({_1x}) = \{ u \in  (T \uplus  {_2X}) \mid   u <_{\mathcal F} ({_1x}) \}.\]
Given ${_1x}, {_1y} \in \os({\mathcal F})$ we write $({_1x}) \sim ({_1y})$ whenever $L({_1x}) = L({_1y})$.
Since ${_1x}$ is comparable in order $<_{\mathcal F}$ to every element of $T \uplus  {_2X}$,
this is an equivalence relation on $\os({\mathcal F})$.
If $[{_1x}]_\sim$ and  $[{_1y}]_\sim$ are distinct classes of the equivalence relation $\sim$, then 
we compare them as follows:

\begin{itemize}

\item {If for some  $u \in  (T \uplus  {_2X})$  we have  $({_1x})<_{\mathcal F} u$ and $u <_{\mathcal F} ({_1y})$
then we define $([{_1x}]_\sim) <_{\mathcal F} ([{_1y}]_\sim)$;}

\item {If for some  $u \in  (T \uplus  {_2X})$  we have  $({_1y})<_{\mathcal F} u$ and $u <_{\mathcal F} ({_1x})$
then we define $([{_1y}]_\sim) <_{\mathcal F} ([{_1x}]_\sim)$;}

\end{itemize}

In a similar way, we define and compare the equivalence classes  $[{_2x}]_\sim$ and  $[{_2y}]_\sim$ for every $x, y \in X_n$.
Now if we define $<_{\mathcal F}$ inside each equivalence class in an arbitrary way, we turn $<_{\mathcal F}$ into a total order on  $\os({\mathcal F})$.
\end{proof}

\begin{question} \label{Q} Given a finite set of words $W$ where each variable occurs at most $m>0$ times, 
is it true that $M(W)$ is finitely related if and only if $M(W \cup \{a^m\})$ is finitely related?
\end{question}

It seems that the answer to Question~\ref{Q} remains unknown if `finitely related' is replaced by `finitely based'.

\begin{cor} \label{C: dis}  Suppose that every word ${\bf u}$ in  $W$ depends on two variables
and every variable occurs twice in $\bf u$. Then $M(W)$ is finitely related if and only if
modulo renaming variables, $W = \{abab, ab^2a, a^2b^2\}$ or $W = \{a^2b^2\}$.

\end{cor}

\begin{proof} In view of Lemma~\ref{L: iso}, renaming letters in $W$ does not change the identities of $M(W)$. So, we may assume that
$W$ is a subset of $\{abab, ab^2a, a^2b^2\}$. Using Lemma~\ref{L: iso}, it is easy to show that $M(abab, ab^2a, a^2b^2)$ is equationally equivalent to $M(W_2)$. Hence these monoids are finitely related by Theorem~3.2 in \cite{DG}. The monoid $M(a^2b^2)$ is finitely related by Example~\ref{E: aabb}.

The following monoids are non-finitely related: 
$M(abab)$ by \cite[Example~4.6]{DG}, 
$M(abab, a^2b^2)$ by \cite[Example~4.14]{DG}, 
$M(abab, ab^2a)$ by \cite[Theorem~4.4]{DG}, 
$M(ab^2a)$ by \cite[Example~4.19]{DG} and
$M(a^2b^2, ab^2a)$ by Example~\ref{E: aba}.
\end{proof}

Notice that Corollary~\ref{C: dis} remains true if we replace `finitely related' by   `finitely based' (see \cite{JS}).
In general, there are no correlation between finitely based and finitely related algebras.
According to Theorem~6.7 in \cite{DG}, the monoid $M(asabtb)$ is finitely related. In contrast, $M(asabtb)$ is non-finitely
based by Lemma~5.5 in  \cite{MJ05}.
While the first examples of non-finitely related algebras were finitely based \cite{Post}, at the moment, no example of a finitely based but not finitely related semigroup exists.

\section*{Acknowledgments}
The author thanks the anonymous referee for helpful comments.


\small



\begin{thebibliography}{123}

\addcontentsline{toc}{section}{{\noindent\bf Bibliography}}

\bibitem{DJPS}  Davey, B.A., Jackson, M.G., Pitkethly, J.G., Szabó, C.: Finite degree: algebras in general and semigroups in particular. Semigroup Forum 83(1), 89–110 (2011) 


\bibitem{DG} Glasson, D.: Finitely and non-finitely related words, Semigroup Forum, 109, 347--374 (2024)

\bibitem{DG1} Glasson, D.: The Rees quotient monoid generates a variety with uncountably many subvarieties, Semigroup Forum, 109, 476--481 (2024)

\bibitem{SG} Gusev, S.V.: Small monoids generating varieties with uncountably many subvarieties. Semigroup Forum, 110, 255--259 (2025)

\bibitem{Gusev-Vernikov}
Gusev, S. V., Vernikov, B. M.: Chain varieties of monoids: {Dissertationes Math.} 534: 1--73 (2018) DOI: 10.4064/dm772-2-2018. 


\bibitem{Gusev-Li-Zhang}
Gusev, S. V., Li, Y. X., Zhang W. T.: Limit varieties of monoids satisfying a certain identity, Algebra Colloquium, 32(01), 1--40 (2025) 


\bibitem{MJ05}  Jackson, M.: Finiteness properties of varieties and the restriction to finite algebras. Semigroup Forum
70(2), 159–187 (2005)


\bibitem{JS}  Jackson,  M. G., Sapir, O. B.: Finitely based, finite sets of words, Internat. J. Algebra Comput., {\bf 10}(6),  683--708 (2000)



\bibitem{JL}  Jackson, M., Lee, E.W.H.: Monoid varieties with extreme properties. Trans. Am. Math. Soc. 370(7),
4785–4812 (2018)



\bibitem{MJ} Jackson, M. G.: {Infinite irredundant equational axiomatisability for a finite monoid}, manuscript, (2015)  arXiv:1511.05979 [math.LO]



\bibitem{Lee14}  Lee, E.W.H.: On certain Cross varieties of aperiodic monoids with commuting idempotents. Results
Math. 66(3–4), 491–510 (2014)



\bibitem{Lee15} Lee, E.W.H.: Inherently non-finitely generated varieties of aperiodic monoids with central idempotents.
Zap. Nauchn. Sem. S.-Peterburg. Otdel. Mat. Inst. Steklov. (POMI), 423 (Voprosy Teorii Predstavleni˘ı
Algebr i Grupp. 26), 166–182 (2014); English transl., J. Math. Sci. (N.Y.) 209(4), 588–599 (2015)




\bibitem{LZ}  Lee E. W. H., Zhang W. T.: The smallest monoid that generates a non-Cross variety (in Chinese), Xiamen Daxue Xuebao Ziran Kexue Ban 53, 1--4 (2014)


\bibitem{Mayr} Mayr, P.: On finitely related semigroups. Semigroup Forum 86(3), 613–633 (2013)

\bibitem{P}  Perkins, P.: Bases for equational theories of semigroups. J. Algebra 11, 298–314 (1969)

\bibitem{Post}  Post, E. L.: The two-valued iterative systems  of mathematical logic. Princeton University Press, Princeton (1941)



\bibitem{RJZL}  Ren, M., Jackson, M., Zhao, X., Lei, D.: Flat extensions of groups and limit varieties of additively
idempotent semirings. J. Algebra 623, 64–85 (2023)

\bibitem{Steindl}  Steindl, M.: Not all nilpotent monoids are finitely related. Algebra Universalis 85(1), 13 (2024)



\end{thebibliography}
\end{document}